\documentclass[12pt]{amsart}
\usepackage[colorlinks=true,pagebackref,hyperindex,citecolor=magenta,linkcolor=Blue]{hyperref}
\usepackage[top=1in, bottom=1in, left=1.1in, right=1.1in]{geometry}
\usepackage{amsmath}
\usepackage{amsfonts}
\usepackage{amssymb}
\usepackage{amsthm}
\usepackage{stmaryrd}
\usepackage{braket}
\usepackage[all]{xy}
\usepackage{mathrsfs}
\usepackage{enumitem} 
\usepackage[T1]{fontenc}
\usepackage{color}
\usepackage{tikz-cd}
\usepackage{tikz}
\usetikzlibrary{matrix,arrows}

\makeatletter
\def\namedlabel#1#2{\begingroup
#2%
\def\@currentlabel{#2}%
\phantomsection\label{#1}\endgroup
}
\makeatother

\theoremstyle{theorem} 
\newtheorem{theorem}{Theorem}[section]

\newtheorem{corollary}[theorem]{Corollary}

\newtheorem{observation}[theorem]{Observation}
\newtheorem{lemma}[theorem]{Lemma}
\newtheorem{proposition}[theorem]{Proposition}

\newtheorem{notation}[theorem]{Notation}

\newtheorem*{convention}{Convention}

\newtheorem{theoremx}{Theorem}

\theoremstyle{definition} 
\newtheorem{definition}[theorem]{Definition}

\newtheorem{example}[theorem]{Example}
\newtheorem{remark}[theorem]{Remark}
\numberwithin{equation}{subsection}

\renewcommand{\(}{\left(}

\newcommand{\NN}{\mathbb{N}}
\newcommand{\RR}{\mathbb{R}}
\newcommand{\ZZ}{\mathbb{Z}}
\newcommand{\PP}{\mathbb{P}}

\newcommand{\CC}{\mathbb{C}}

\newcommand{\cA}{\mathscr{A}}

\newcommand{\IN}{\operatorname{in}}

\makeatletter
\def\@tocline#1#2#3#4#5#6#7{\relax
  \ifnum #1>\c@tocdepth 
  \else
    \par \addpenalty\@secpenalty\addvspace{#2}%
    \begingroup \hyphenpenalty\@M
    \@ifempty{#4}{%
      \@tempdima\csname r@tocindent\number#1\endcsname\relax
    }{%
      \@tempdima#4\relax
    }%
    \parindent\z@ \leftskip#3\relax \advance\leftskip\@tempdima\relax
    \rightskip\@pnumwidth plus4em \parfillskip-\@pnumwidth
    #5\leavevmode\hskip-\@tempdima
      \ifcase #1
       \or\or \hskip 1.9em \or \hskip 2em \else \hskip 3em \fi%
      #6\nobreak\relax
    \dotfill\hbox to\@pnumwidth{\@tocpagenum{#7}}\par
    \nobreak
    \endgroup
  \fi}
\makeatother

\newcommand{\Spec}{\operatorname{Spec}}
\newcommand{\Frac}{\operatorname{Frac}}

\newcommand{\Tor}{\operatorname{Tor}}
\newcommand{\Min}{\operatorname{Min}}

\newcommand{\Ext}{\operatorname{Ext}}
\newcommand{\Proj}{\operatorname{Proj}}
\newcommand{\Ann}{\operatorname{Ann}}
\newcommand{\Ht}{\operatorname{ht}}
\newcommand{\height}{\operatorname{ht}}

\newcommand{\p}{\mathfrak{p}}
\newcommand{\q}{\mathfrak{q}}
\renewcommand{\a}{\mathfrak{a}}

\newcommand{\m}{\mathfrak{m}}
\newcommand{\n}{\mathfrak{n}}
\newcommand{\ara}{\operatorname{ara}}
\newcommand{\cdim}{\operatorname{c}}

\definecolor{blue-violet}{rgb}{0.54, 0.17, 0.89}
\definecolor{Blue}{rgb}{0.01, 0.28, 1.0}
\definecolor{gGreen}{rgb}{0.2, 0.8, 0.2}
\definecolor{Green}{rgb}{0.04, 0.85, 0.32}

\begin{document}

\title[Connectedness of square-free Groebner Deformations]{Connectedness of square-free Groebner Deformations}

\author[L. Alan\'is-L\'opez]{Lilia Alan\'is-L\'opez}
\address{Departamento de Matem\'aticas\\
Universidad Aut\'onoma de Nuevo Le\'on, M\'exico} \email{lilia.alanislpz@uanl.edu.mx}

\author[L. N\'u\~{n}ez-Betancourt]{Luis N\'u\~{n}ez-Betancourt{$^1$}}
\address{Departamento de Matem\'aticas, Centro de Investigaci\'on en Matem\'aticas, M\'exico}
\email{luisnub@cimat.mx}

\author[P. Ram\'irez-Moreno]{Pedro Ram\'irez-Moreno{$^2$}}
\address{Departamento de Matem\'aticas\\
Universidad Aut\'onoma de Nuevo Le\'on, M\'exico}
\email{pedro.ramirezmrn@uanl.edu.mx}

\thanks{{$^1$}This author was partially supported by  CONACyT Grant 284598 and C\'atedras Marcos Moshinsky}
\thanks{{$^2$}This author was partially supported by CONACyT Grants  706865 and 284598.}

\subjclass[2010]{Primary 13E05, 05E40, 13C15, 13H99, 13J10; Secondary 13B25,13P10.}
\keywords{Connectedness Dimension; Noetherian Equidimensional Complete Local Rings; Gr\"{o}bner Bases; Monomial Orders; Non zero divisors.}

\maketitle

\begin{abstract}
Let $I\subseteq S=K[x_1,\ldots,x_n]$ be a homogeneous ideal equipped with a monomial order $<$.
We show that if $\IN_{<}(I)$ is a square-free monomial ideal, then $S/I$ and $S/\IN_<(I)$ have the same connectedness dimension. We also show that graphs related to connectedness of these quotient rings have the same number of components. 
We also provide consequences regarding Lyubeznik numbers.
We obtain these results by furthering the study of  connectedness modulo a  parameter 
in a  local ring.

\end{abstract}

\setcounter{tocdepth}{1}
\tableofcontents


Let $I\subseteq S=K[x_1,\ldots,x_n]$ be a homogeneous ideal in a standard graded polynomial ring over a field. 
A theme of research is to obtain properties of $I$ from the monomial ideal $\IN_<(I)$. For instance, it is well known that these two ideals have the same Hilbert function, and so, the same dimension.

In this work we focus on the connectedness dimension, $\cdim(X)$, of an algebraic variety, $X$. This number measures how connected $X$ is.
Explicitly,
$$
\cdim(X)=\min\{ \dim(Z)\;|\; Z\hbox{ is closed and }X\setminus Z \hbox{ is diconnected}\}.
$$
This line of research was initiated by 
 Varbaro \cite{MV09}, who  showed that  $c(S/\IN_<(I))\geq c(S/I)$ (see also \cite{KS}).
If $X=\Spec(R)$ is an equidimensional affine variety, there is a sequence of graphs, $\Gamma_t(R)$, that detect $c(X)$ (see Definition \ref{DefGraph}). Explicitly,
$$
\# \Gamma_t(R)=\max\{ \#(X\setminus Z)\; |\; Z\hbox{ is a closed subvariety}\},
$$
where $\#$ denotes the number of connected components \cite{NBSW}.
We recall that $\Gamma_1(R)$ is called the dual graph \cite{HCT} or  Hochster-Huneke graph \cite{HH} of $R$. 
Similarly, $\Gamma_t(R)$ played a role in the study of vanishing of local cohomology \cite{HVT,HL,O}.

Recently, Conca and Varbaro \cite{CV} showed that if $\IN_<(I)$ is a square-free monomial ideal  more properties of $I$ are reflected in $\IN_<(I)$. In particular, they showed that the $a$-invariants and extremal Betti numbers of $\IN_<(I)$ and $I$ are equal. In this work, we show that if $\IN_<(I)$ is a square-free monomial ideal, then  $c(S/\IN_<(I))=c(S/I)$.

\begin{theoremx}[{Theorem \ref{ThmGroeb}}]\label{ThmMainGroeb}
Let $S=K[x_1,\ldots,x_n]$ be a polynomial ring over a field.
Let $I$ be an  equidimensional ideal and $<$ a monomial order such that $\dim S/I\geq 2$ and  $\IN_<(I)$ is square-free.
Then, 
$$
\#\Gamma_t(S/I)=\#\Gamma_t(S/\IN(I)).
$$
As a consequence,
$\cdim(S/I)=\cdim(S/\IN(I))$.
\end{theoremx}

Theorem \ref{ThmMainGroeb} extends previous work of Nadi and Varbaro \cite{NV} in which they showed that 
$\Gamma_t(S/I)$ is connected if and only if $\Gamma_t(S/\IN(I))$ is connected. We show that, under the conditions in Theorem \ref{ThmMainGroeb},  $\#\Gamma_t(S/I)$ does not change after field extensions   (see Corollary \ref{CorField}).

The graphs $ \Gamma_t(R)$ have played an important role in the study of Lyubeznik numbers \cite{LSI,NBSW,W,Z}, which are numerical invariants obtained from local cohomology \cite{L}. We refer to a survey on Lyubeznik numbers for more information about this topic \cite{SLN}. As a consequence of Theorem \ref{ThmMainGroeb}, we are able show that certain Lyubeznik numbers of $S/I$ and $S/\IN_<(I)$ are equal if $\IN_<(I)$ is square-free (see Theorem \ref{ThmLyuNum}).

The main technique to show Theorem \ref{ThmMainGroeb} is the study of connectedness under a parameter ideal. 
In fact, Theorem \ref{ThmMainGroeb} is an application of the following result.

\begin{theoremx}[{Theorem \ref{ThmGraphParameter}}]\label{thm:Main2}
Let $(R,\mathfrak m)$ be an equidimensional complete local ring containing a field of $\dim(R) = d \geq 3$, with separably closed residue field. Suppose there exists $x \in \mathfrak m$ such that $x$ is a non zero divisor of $R$ and $(x)$ is a radical ideal. Let $t$ be an integer such that  $1\leq t\leq d-2$. Then,
$$ \#\Gamma_t(R) = \#\Gamma_t(R/(x)).$$
As a consequence,
$$ \cdim(R) = \cdim(R/(x)) + 1.$$
\end{theoremx}

Theorem \ref{thm:Main2} is on the same line of research studied by Spiroff, Witt, and the second-named author \cite{NBSW}, where different conditions were imposed on the parameter to obtain the equality of connectedness dimension.

\begin{convention}
All rings in this manuscript are commutative Noetherian with one.
\end{convention}

\section{Graphs and Connectedness Dimension}
We start by recalling the  construction  of certain  graphs that detect the connectedness of a variety or a ring.

\begin{definition}[{\cite{NBSW}}]\label{DefGraph}
Let $R$ be a  ring of dimension $d$ and let $t$ be an integer such that  $0\leq t\leq d$. We define the graph $\Gamma_t(R)$ as a simple graph whose vertices are the minimal primes of $R$ and there is an edge between $\p$ and $\q$ distinct minimal primes if and only if $\Ht(\p + \q ) \leq t$.
\end{definition}

We note that $\Gamma_s(R)$ is a subgraph of $\Gamma_t(R)$ for every $s$ and $t$ such that $0 \leq s \leq t \leq d$.  As a consequence, if $\Gamma_s(R)$ is connected, then $\Gamma_t(R)$ is also connected. We observe that $\Gamma_0(R)$ is connected if and only if $R$ has only one minimal prime.  In addition,  $\Gamma_d(R)$ is  connected if $R$ is local.

Connectedness dimension is a ring invariant and one of our main objects of study. We define it in the following way: 

\begin{definition}
Let $R$ be a ring. We define the connectedness dimension of $R$, and denote it as $\cdim(R)$, as 
$$ \cdim(R) = \min\Set{ \dim (R/I) | \Spec(R) - V(I) \text{ is disconnected }} $$
We take the convention that the empty set is disconnected.
\end{definition}

\begin{theorem}\label{Grothendieck}
Let $R$ be an equidimensional complete local ring. Let $I$ be a proper ideal of $R$. Then
$$ \cdim(R/I) \geq \min\Set{\cdim(R), \dim(R) - 1} - \ara(I) $$ 
\end{theorem}

We state some  well-known properties to have  that are helpful in several results.

\begin{remark}[{\cite[Remark 2.6]{NBSW}}]\label{rmk:height}
Let $R$ be a ring and let $I_1, \dots, I_n, J_1, \dots, J_m$ be ideals of  $R$, then
\begin{enumerate}
\item $\sqrt{\bigcap_{i = 1}^{n}I_i + \bigcap_{j = 1}^{m}J_j} = \sqrt{\bigcap_{i = 1}^{n}\bigcap_{j = 1}^{m}(I_i + J_j)}.$
\item $\Ht(\bigcap_{i = 1}^{n}I_i + \bigcap_{j = 1}^{m}J_j) = \min\Set{\Ht(I_i + J_j) | 1 \leq i \leq n, 1 \leq j \leq m}.$
\end{enumerate}
\end{remark}

\begin{remark}\label{rmk:ht+dim}
Let $A$ be a  complete local equidimensional ring and let $I$ be an ideal of $A$. Then
$$\height(I) + \dim(A/I) = \dim(A).$$
\end{remark}

A graph is connected if and only if no matter how we partition its set of indices in two non empty sets, we can always find an edge between a vertex of one of these two disjoint sets and a vertex of the other one. 

\begin{proposition}\label{pr:GraphConnectedPartitionEquiv}
Let $(R, \mathfrak{m})$ be a  local ring of dimension $d$ with more than one minimal prime. Let $t$ be an integer such that  $1\leq t\leq d-1$. Then $\Gamma_t(R)$ is connected if and only if $\height( \bigcap_{\p \in S}\p + \bigcap_{\q \in T}\q) \leq t$ for every $(S, T)$ partition of $\Min(R)$ such that $S$ and $T$ are non empty.
\end{proposition} 

\begin{proof}
By Remark \ref{rmk:height}, we know that given a $(S, T)$ partition of $\Min(R)$, we have that
$$\Ht\left(\bigcap_{\p \in S}\p + \bigcap_{\q \in T}\q\right) = \min\Set{\Ht(\p + \q)|\p \in S, \q \in T},$$
so there must be $\p' \in S$ and $\q' \in T$ such that $\Ht\left(\bigcap_{\p \in S}\p + \bigcap_{\q \in T}\q\right) = \Ht(\p' + \q')$.
This means that for every $(S, T)$ partition of $\Min(R)$ such that $S$ and $T$ are non empty , we have that
$$\height\( \bigcap_{\p \in S}\p + \bigcap_{\q \in T}\q\right) \leq t \Leftrightarrow \exists \p' \in S, \q' \in T : \Ht(\p' + \q') \leq t.$$
So, for any such partition $(S, T)$, you can find an edge between $S$ and $T$. This happens if  and only if $\Gamma_t(R)$ is connected.
\end{proof}

\begin{remark}\label{rmk:GoingModuloParameter}
Let $(R,\mathfrak{m})$ be an equidimensional local ring with $\dim(R) = d \geq 1$. Let $x \in \mathfrak{m}$ such that $x$ is not an element of any minimal prime of $R$. 
Then, $\dim R/(x)=d-1$, because $x$ is a parameter. 
We know that the minimal primes of $R/(x)$ are of the form $\q/(x)$, with $\q$ a minimal prime of $(x)$. 
 Then $\dim\( \frac{R/(x)}{\q /(x)} \right) = \dim(R/\q)$, and $\Ht(\q)=1$ by Krull's Principal Ideal Theorem. We know that in $R$ the equality $\Ht(\q) + \dim(R/\q) = \dim(R) = d$ holds from Remark \ref{rmk:ht+dim}, we conclude that $\dim(R/\q) = d - 1$. So $\dim\(  \frac{R/(x)}{\q /(x)} \right) = d - 1$. This means $R/(x)$ is  equidimensional.
\end{remark}

Given a $\Gamma$ graph we focus our attention in its subgraph corresponding to certain subset of minimal primes. One way to study such subgraph is by doing specific quotients of the ring.

\begin{proposition}\label{pr:QuotientRingGraphEquiv}
Let $(R,\mathfrak m)$ be an equidimensional complete local ring and let $I$ be a proper ideal of $R$ such that $\Min(R/I) \subseteq \Min(R) $. Then $R/I$ is also an equidimensional complete local ring and $\dim(R/I) = \dim(R)$. Furthermore, if $J$ is an ideal of $R$ such that $I \subseteq J$ then $\height(J) = \height(J/I)$. In addition, if $\Sigma$ is the subgraph of $\Gamma_t(R)$ whose vertices are $\Min(I)$, then
$$ \Sigma \cong \Gamma_t{(R/I)}.$$
\end{proposition}

\begin{proof}
We know that $R/I$ is a  complete local ring. We also know that the minimal primes of $R/I$ are the ideals of the form $\p/I$ with $\p$ minimal prime of $R$.

Observe that $\dim\left(\frac{R/I}{\p /I}\right) = \dim(R/\p) = \dim(R)$ for every minimal prime $\p$ of $R$ since $R$ is equidimensional. This means $R/I$ is also equidimensional and $\dim(R/I) = \dim(R)$.

Let $J$ be an ideal of $R$ such that $I \subseteq J$. By Remark \ref{rmk:ht+dim} we have that $\Ht(J/I) + \dim\left(\frac{R}{I}/\frac{J}{I}\right) = \dim(R/I)$, so $\Ht(J/I) = \dim(R) - \dim(R/J)$. Remark \ref{rmk:ht+dim} implies that $\Ht(J) = \dim(R) - \dim(R/J)$. Thus $\Ht(J) = \Ht(J/I)$.

The correspondence between vertices of $\Sigma$ and vertices of $\Gamma_{t}(R/I)$ is given by assigning each minimal prime $\p$ of $I$ to the minimal prime $\p/I$ of $R/I$. Thus the vertices are preserved. Notice that edges are also preserved since if there is a edge between $\p$ and $\q$ minimal primes of $I$, then $\Ht(\p + \q) \leq t$. This is the same as saying that $\Ht(\p/I + \q/I) = \Ht((\p + \q)/I ) \leq t$ since $\Ht(\p + \q) = \Ht((\p + \q)/I )$ by the previous paragraph. 
\end{proof}

In particular we choose $I$ to be exactly the intersection of the minimal primes corresponding to the part of the graph we want to focus our attention on.

The next proposition gives us more information about how the graphs work when we study the quotient ring with different ideals but with the same radical.

\begin{observation}\label{obs:GammaGraphQuotientRadical}
Let $(R,\mathfrak m)$ be an equidimensional complete local ring and let $I$, $J$ be ideals of $R$ such that $\sqrt{I} = \sqrt{J}$. Then both $R/I$ and $R/J$ are  complete local rings of the same dimension and if $R/I$ is equidimensional, then $R/J$ is also equidimensional and $ \Gamma_t(R/I) \cong \Gamma_t{(R/J)} $.
\end{observation}


Now we are ready to begin exploring the relations between connectedness dimension and the $\Gamma$ graphs.

\begin{proposition}[{\cite[Proposition 2.5]{NBSW}}]\label{pr:GraphCDimEquiv}
Let $(R, \mathfrak{m})$ be an equidimensional complete local ring with $\dim(R) = d \geq 2$. let $t$ be an integer such that $1\leq t\leq d-1$. Then
$$ \Gamma_{t}(R) \text{ is connected } \Leftrightarrow \cdim(R) \geq d - t.$$
As a consequence, the connectedness dimension is given by
$$ \cdim(R) = \max\Set{ i \;|\; \Gamma_{d-i}(R) \text{ is connected }}.$$
\end{proposition}

We can compute connectedness dimension by counting how many of the $\Gamma_{t}(R)$ graphs with $t \in [0, d-1]$ are connected. 
Even if a graph $\Gamma_t(R)$ is not connected we can also obtain information regarding its connected components.

\begin{notation}
Let $G$ be a graph and let $X$ be a topological space. We denote $\#G$ to the amount of connected components of $G$ and denote $\#X$ to the amount of connected components of the space $X$.
\end{notation}

\begin{corollary}[{\cite[Corollary 2.7]{NBSW}}]\label{cor:NumberConnectedComponents}
Let $(R, \m)$ be an equidimensional complete local ring of dimension $d \geq 2$.
Let $t$ be an integer such that $1\leq t\leq d-1$.
Then: 
$$ \#\Gamma_t(R) = \max\Set{\#(\Spec(R) - V(I))\; |\; \dim(R/I) < d - t }.$$
\end{corollary}

\section{Connectedness Dimension and $\Gamma_t$ Graphs Modulo a Parameter}
In this section we study our connectedness graphs
 modulo a parameter. In order to do this, we need to develop need additional tools. The following lemma gives us information about the behavior between minimal primes of a ring and the minimal primes of an ideal generated by a parameter. 

\begin{lemma}\label{lem1}
Let $(R,\mathfrak m)$ be an  equidimensional complete local ring with $\dim(R) = d \geq 1$. Let $x \in \mathfrak m$ such that $x$ is not an element of any minimal prime of $R$. Then
\begin{enumerate}
\item For every minimal prime $\q$ of $(x)$, there is a minimal prime $\p$ of $R$ such that $\p \subseteq \q$.
\item For every minimal prime $\p$ of $R$, there is a minimal prime $\q$ of $(x)$ such that $\p\subseteq \q$.
\end{enumerate}
\end{lemma}

\begin{proof}
For the first part, we observe that $\Ht(\q)=1$. Then, there must be a minimal prime $\p$ of $A$ such that $\p \subseteq \q$ by prime avoidance.

We now show the second part.
Let $\p$ be a minimal prime of $R$. Notice that $R/\p$  that $\overline{x}$ is not contained in the unique minimal prime of $R/\p$, so by Remark \ref{rmk:GoingModuloParameter}, we get that $\frac{R  / \p}{(\overline{x})} \cong R/(\p+(x))$ is equidimensional of dimension $d - 1$ and that $\height(\overline{x})=1$.

Since 
$$ \height(\overline{x})+\dim\left(\frac{R/\p}{(\overline{x})}\right)=\dim\left(R /\p\right),$$
we conclude that 
$$ \dim\left(R/(\p+(x))\right) = d - 1.$$
We also know that
$$ \height(\p + (x)) + \dim\left(R/(\p+(x))\right) = \dim(R),$$
and so,
$$ \height(\p + (x)) = 1.$$
Now take $\q \in \Min(\p + (x))$ such that $\height(\q) = \height(\p + (x))$. Since $(x) \subseteq \p + (x)$ and $\height(x) = \height(\p + (x))$ by Remark \ref{rmk:GoingModuloParameter}, then $\q \in \Min(x)$. Finally $\p \subseteq \q$ because $\q \in \Min(\p + (x))$.  
\end{proof}

From  Lemma \ref{lem1}, we are able to characterize minimal primes of $R/(x)$.  

\begin{proposition}\label{prop:D(P)=Min(P+(x))}
Let $(R,\mathfrak m)$ be an equidimensional complete local ring with $\dim(R) = d \geq 1$. Let $x \in \mathfrak m$ such that $x$ is not an element of any minimal prime of $R$. For every minimal prime $\p$ of $R$, we have that
$$ \Min(\p + (x)) = \{\q \in \Min(x)\:|\:\p \subseteq \q\}.$$
\end{proposition}

\begin{proof}
We proceed by double containment. Let $ \cA = \{\q \in \Min(x)\:|\:\p \subseteq \q\}$.

Take a minimal prime $\q$ of $(x)$ that contains $\p$. Since $\p + (x) \subseteq \q$ and  $\height(\p + (x)) = \height(\q)$, then $\q$ must be a minimal prime of $\p + (x)$. Then, $\cA \subseteq \Min(\p + (x))$.

Now, let $\p \in \Min(R)$ and let $\q \in \Min(\p + (x))$. We show that $\height(\q)=1$. From Remark \ref{rmk:GoingModuloParameter}, we know that $\frac{R}{\p + (x)}$ is an equidimensional ring of dimension $d-1$. We know from Remark \ref{rmk:ht+dim} that
$$ 
\height\left(\q/(\p + (x))\right) + \dim\left(\frac{R/(\p + (x))}{\q/(\p + (x))}\right) = \dim\left(R/(\p + (x))\right).$$
Since $\q/(\p + (x)) \in \Min\left(R/(\p + (x))\right)$, we get that $\height\left(\q/(\p + (x))\right) = 0$. We conclude that 
$$ \dim\left(\frac{R/(\p + (x))}{\q/(\p + (x))}\right) 
= \dim\left(R/\q\right) = d - 1.$$
We also know that $\height(\q) + \dim(R/\q) = \dim(R)$, and so $\height(\q) = 1$.
This means that $\q$ is a minimal prime of $(x)$ and contains $\p$. Hence,  $\Min(\p + (x)) \subseteq \cA$.
\end{proof}

We know $D(\p) = \Min(\p + (x))$ by Proposition \ref{prop:D(P)=Min(P+(x))}. From Lemma \ref{lem1}, we deduce that $\bigcup_{\p \in \Min(R)}D(\p) = \Min(x)$.

\begin{corollary}\label{cor:MinInt+(x)}
Let $(R,\m)$ be an equidimensional complete local ring of $\dim(R) = d \geq 1$. Let $x \in \m$ such that $x$ is a not an element of  any minimal prime of $R$. Let $S$ be a non empty subset of $\Min(R)$ and let $I = \bigcap_{\p \in S}\p$. Then
$$ \Min(I + (x)) = \bigcup_{\p \in S}\Min(\p + (x)).$$
\end{corollary}

\begin{proof}
We proceed by double containment. 

First we prove that $ \Min(I + (x)) \subseteq \bigcup_{\p \in S}\Min(\p + (x))$. Take a minimal prime $\q$ of $I + (x)$. By prime avoidance $\q$ contains a prime $\p \in S$. But $I + (x) \subseteq \p + (x) \subseteq \q$. This implies that $\q$ is also a minimal prime of $\p + (x)$. This also means that $\height(I + (x)) = 1$, because all the minimal primes of $\p + (x)$ are of height $1$ by Proposition \ref{prop:D(P)=Min(P+(x))}.

Now, we prove that $\bigcup_{\p \in S}\Min(\p + (x)) \subseteq  \Min(I + (x))$. Let $\p \in S$ and $\q$ be a minimal prime of $\p + (x)$. Since $I + (x) \subseteq \p + (x)$ and $\height(I + (x)) = \height(\p + (x)) = \height(\q)$, we deduce that $\q$ must also be a minimal prime of $I + (x)$.
\end{proof}

The following definition plays a key role in the rest of the section, in particular, in  Theorem \ref{thm:Main2}.

\begin{definition}
Let $(R,\mathfrak m)$ be an equidimensional complete local ring with $\dim(R) = d \geq 1$. Let $x \in \mathfrak m$ be such that $x$ is not an element of any minimal prime of $R$. Given a minimal prime $\p$ of $R$, we define the dust of $\p$ modulo $x$ by 
$$D_x(\p)=\{\q \in \Min(x)\:|\:\p \subseteq \q\}.$$
Furthermore if $\Sigma$ is a subgraph of $\Gamma_t(R)$, then
$$D_x(\Sigma)=\bigcup_{\p \in \Sigma}D_x(\p).$$
If $x$ is clear from the context, we omit the subscript.
\end{definition}

\begin{definition}
Let $(R,\mathfrak m)$ be an  equidimensional complete local ring with $\dim(R) = d \geq 1$. Let $x \in \mathfrak m$ be such that $x$ is not an element of any minimal prime of $R$. Let $\Sigma$ be a subgraph of $\Gamma_t(R)$. Let $\Sigma^{'}$ be the subgraph of $\Gamma_t(R/(x))$ such that its vertices are given by $\q/(x)$ such that $\q \in D(\Sigma)$. We call $\Sigma^{'}$ the associated graph to $\Sigma$.
\end{definition}

In the previous setting let $\Sigma$ be the subgraph of $\Gamma_t(R)$ whose vertices are the elements of $S$. We  observe that $D(\Sigma) = \Min(I + (x))$, because $D(\p) = \Min(\p + (x))$.

Now we are ready for our study of connectedness dimension modulo a parameter. It turns out that if $\Gamma_t(R)$ is connected, then $\Gamma_t(R/(x))$ is also connected. The only case when this is not necessarily true is when $t = 0$, as the following example shows. 

\begin{example}
Let $K$ be a field and consider the power series ring $R = K[[x,y,z]]$. $\Gamma_t(R)$ is connected for every $t$ since $R$ is a domain, but $\Gamma_0(R/(xyz))$ is not connected since it has more than one minimal prime. 
\end{example}

Additionally, we restrict $t$ to be less or equal than $d-2$. We do so because $\Gamma_{d-1}(R/(x))$ is connected regardless the connectedness of $\Gamma_{d-1}(R)$.
We now recall a result on this regarding connectedness modulo a parameter. We poit out that the original statement requires that the residue field is separably closed \cite[Proposition 3.2]{NBSW}; however, this hypothesis is not necessary in the proofs.

\begin{theorem}[{\cite[Proposition 3.2]{NBSW}}]\label{thm1}
Let $(R,\m)$ be an equidimensional complete local ring containing a field,of $\dim(R) = d \geq 3$. Let $x \in \m$ such that $x$ is a not an element of  any minimal prime of $R$. Let $t$ be an integer such that  $1\leq t\leq d-2$. Then,
$$\Gamma_t{(R)} \:\text{is connected} \Rightarrow \Gamma_{t}(R/(x)) \:\text{is connected}.$$
\end{theorem}

We now give a slightly more general  version of Theorem \ref{thm1}.
 This is needed in our proof of Theorem \ref{thm:Main2}.

\begin{corollary}\label{cor:ConnectedImpSigmaSigma'}
Let $(R,\m)$ be an equidimensional complete local ring containing a field, of $\dim(R) = d \geq 3$. Let $x \in \m$ such that $x$ is a not an element of  any minimal prime of $A$. Let $t$ be an integer such that $1\leq t\leq d-2$. Let $\Sigma$ be a subgraph of $\Gamma_t(R)$ and let $\Sigma^{'}$ be the subgraph of $\Gamma_t(R/(x))$ associated to $D(\Sigma)$. Then,
$$ \Sigma \:\text{is connected} \Rightarrow \Sigma^{'} \:\text{is connected}.$$
\end{corollary}

\begin{proof}
Suppose $\Sigma$ is connected. Let $I$ be the intersection of all the vertices of $\Sigma$. From Proposition \ref{pr:QuotientRingGraphEquiv} we know that $\Sigma \cong \Gamma_t(R/I)$, so $\Gamma_t(R/I)$ is also connected and Theorem \ref{thm1} implies that $\Gamma_t\left( R/(I + (x)) \right )$ is also connected.

From Corollary \ref{cor:MinInt+(x)} we know that $\sqrt{I + (x)} = \bigcap_{\q \in \Min(I + (x))}\q = \bigcap_{\q \in D(\Sigma)}\q$. Observation \ref{obs:GammaGraphQuotientRadical} implies that $\Gamma_t\left( R/(I + (x)) \right) \cong \Gamma_t\left (R/\bigcap_{\q \in D(\Sigma)}\q\right )$.
Then,
$$\Gamma_t\left (\frac{R}{\bigcap_{\q \in D(\Sigma)} \q}\right ) \cong \Gamma_t\left(\frac{R/(x)}{\bigcap_{\q \in D(\Sigma)}\q/(x)}\right) \cong \Sigma^{'}$$
by Proposition \ref{pr:QuotientRingGraphEquiv}.
We conclude that $\Sigma^{'}$ is also connected.
\end{proof}

\begin{lemma}\label{lem:SimaSigma'0VrtxCmmn}
Let $(R,\m)$ be an equidimensional complete local ring of $\dim(R) = d \geq 1$. Let $x \in \m$ such that $x$ is not in any minimal prime of $R$. Let $t$ be an integer such that  $1\leq t\leq d-2$. Let $\Sigma_i$ and $\Sigma_j$ be subgraphs of $\Gamma_t(R)$. If the graphs $\Sigma_i$ and $\Sigma_j$ do not share any vertices and there are no edges between them, then $D(\Sigma_i)$ and $D(\Sigma_j)$ are disjoint. In particular the subgraphs $\Sigma_i^{'}$ and $\Sigma_j^{'}$ of $\Gamma_t(R/(x))$ associated to $\Sigma_i$ and $\Sigma_j$ respectively do not share vertices.
\end{lemma}

\begin{proof}
Let $\q \in D(\Sigma_i) \cap D(\Sigma_j)$, then there are $\p_i \in \Sigma_i$ and $\p_j \in \Sigma_j$ such that $\q \in D(\p_i) \cap D(\p_j)$. So $\p_i + \p_j \subseteq \q$, this means that $\Ht(\p_i + \p_j) \leq \Ht(\q) = 1 \leq t$. So there is an edge between $\Sigma_i$ and $\Sigma_j$, a contradiction.
\end{proof}

The following lemma plays the main role in the proof Theorem \ref{thm:Main2}. 

\begin{lemma}\label{lem:Tor}
Let $(R,\m)$ be an equidimensional complete local ring. Suppose there is an $x \in \m$ such that $x$ is a non zero divisor of $R$ and that $(x)$ is a radical ideal. Let $(S, T)$ be a partition of $\Min(R)$ such that $S$ and $T$ are non empty, and $I = \bigcap_{\p \in S}\p$ and $J=\bigcap_{\q \in T}\q$. Then, $x$ is a non zero divisor of $\frac{R}{I + J}$. In particular $x$ is not in any minimal prime of $I + J$.
\end{lemma}

\begin{proof}
Consider the exact sequence 
$$0 \rightarrow \frac{R}{I \cap J} \rightarrow \frac{R}{I} \oplus \frac{R}{J} \rightarrow \frac{R}{I+J} \rightarrow 0.$$
This sequence induces a long exact sequence of Tor of the form

\begin{center}
\begin{tikzcd}
 \cdots \rar & \Tor_1\left( \frac{R}{I \cap J} ,\frac{R}{(x)}\right) \rar & \Tor_{1}\left( \frac{R}{I} \oplus \frac{R}{J},\frac{R}{(x)} \right) \rar
             \ar[draw=none]{d}[name=X, anchor=center]{}
    & \Tor_{1}\left(\frac{R}{I+J},\frac{R}{(x)}\right) \ar[rounded corners,
            to path={ -- ([xshift=2ex]\tikztostart.east)
                      |- (X.center) \tikztonodes
                      -| ([xshift=-2ex]\tikztotarget.west)
                      -- (\tikztotarget)}]{dll}[at end]{} \\      
 & \frac{R}{I \cap J} \otimes \frac{R}{(x)} \rar &\left(\frac{R}{I} \oplus \frac{R}{J}\right)\otimes \frac{R}{(x)} \rar &  \frac{R}{I+J} \otimes \frac{R}{(x)}\rar &  0
\end{tikzcd}
\end{center}

Since $x$ is a non zero divisor of $R$, we have that $\Tor_1(R/I,R/(x)) = \Ann_{R/I}(x)$, and $ \Ann_{R/I}(x)=0$.
Similarly $\Tor_1(R/J,R/(x)) = \Ann_{R/J}(x) = 0$. Then,
$$\Tor_{1}\left(\frac{R}{I} \oplus \frac{R}{J},\frac{R}{(x)}\right) = \Tor_1\left(\frac{R}{I},\frac{R}{(x)}\right) \oplus \Tor_1\left(\frac{R}{J},\frac{R}{(x)}\right) = 0.$$
We have that $\Tor_{1}\left(\frac{R}{I+J},\frac{R}{(x)}\right) = \Ann_{\frac{R}{I+J}}(x)$, and $\frac{R}{\a}\otimes \frac{R}{(x)}=\frac{R}{\a+(x)}$ for every ideal $\a\subseteq R.$
Then, 
$$0 \rightarrow \Ann_{\frac{R}{I+J}}(x) \rightarrow \frac{R}{I \cap J + (x)} \rightarrow \frac{R}{I+(x)} \oplus \frac{R}{J+(x)} \rightarrow \frac{R}{I+J+(x)} \rightarrow 0.$$
Observe that
\begin{eqnarray*}
(x) & \subseteq & \sqrt{0} + (x) \\
& = & I \cap J + (x) \\
& \subseteq & (I+(x)) \cap (J+(x)) \\
& \subseteq & \sqrt{(I+(x)) \cap (J+(x))} \\
& = & \sqrt{I \cap J+(x)} \\
& = & \sqrt{ \sqrt{0} + \sqrt{(x)}} \\
& = & \sqrt{ 0 + (x)} \\
& = & \sqrt{(x)} \\
& = & (x).
\end{eqnarray*}
This  implies that $I \cap J + (x) = (I+(x)) \cap (J+(x))$. Since the sequence
$$0 \rightarrow \frac{R}{(I+(x)) \cap (J+(x))} \rightarrow \frac{R}{I + (x)} \oplus \frac{R}{J + (x)} \rightarrow \frac{R}{(I+(x))+(J+(x))} \rightarrow 0$$
is exact, we conclude that $\Ann_{\frac{R}{I+J}}(x) = 0$. This means that $x$ is a non zero divisor of $\frac{R}{I+J}$,
and  so, $x$ is not in any minimal prime of $I + J$.
\end{proof}

\begin{theorem}\label{thm2}
Let $(R,\mathfrak m)$ be a  equidimensional complete local ring with $\dim(R) = d \geq 3$. Suppose there exists an $x \in \mathfrak m$ such that $x$ is a non zero divisor of $R$ and that $(x)$ is a radical ideal. Let $t$ be an integer such that  $1\leq t\leq d-2$. Then
$$\Gamma_{t}(R/(x)) \:\text{is connected} \Rightarrow \Gamma_t{(R)} \:\text{is connected.}$$
As a consequence
$$ \cdim(R) = \cdim(R/(x)) + 1.$$
\end{theorem}

\begin{proof}
Let $(S, T)$ be a partition  of $\Min(R)$.
 We know $\cdim(R) = \dim(R/(I+J))$ where $I$ and $J$ are the intersection of all the elements of $S$ and $T$ respectively.

From Lemma \ref{lem:Tor}, we know that $x$ is not an element of any minimal prime of $I + J$. Then,
\begin{eqnarray*}
\cdim(R) & = & \dim(R/(I+J)) \\
& = & \dim(R/(I+J + (x))) + 1 \\
& \geq & \dim\left(\frac{R}{\bigcap_{\p \in S}\bigcap_{\q \in D(\p)}\q + \bigcap_{\p \in T}\bigcap_{\q \in D(\p)}\q}\right)+1 \\
& = & \dim\left(\frac{R/(x)}{\bigcap_{\p \in S}\bigcap_{\q \in D(\p)}\q/(x) + \bigcap_{\p \in T}\bigcap_{\q \in D(\p)}\q/(x)}\right)+1 \\
& \geq & \cdim(R/(x)) + 1.
\end{eqnarray*}

From Proposition \ref{pr:GraphCDimEquiv} we have the inequality $\cdim(R/(x)) \geq (d - 1) - t$, so $\cdim(R/x) + 1 \geq d - t$. From our previous chain of inequalities, we get that $\cdim(R) \geq d -t$. We conclude that $\Gamma_t(R)$ is connected.
\end{proof}

\begin{theorem}\label{ThmGraphParameter}
Let $(R,\mathfrak m)$ be an equidimensional complete local ring containing a field, of $\dim(R) = d \geq 3$. Suppose there exists $x \in \mathfrak m$ such that $x$ is a non zero divisor of $R$ and that $(x)$ is a radical ideal. Let $t$ be an integer such that  $1\leq t\leq d-2$. Then
$$ \#\Gamma_t(R) = \#\Gamma_t(R/(x)).$$
\end{theorem}

\begin{proof}
Suppose $\#\Gamma_t(R) = s$. Let $\Sigma_1, \dots, \Sigma_s$ be the $s$ connected components of $\Gamma_t(R)$. Let $\Sigma_1^{'}, \dots, \Sigma_s^{'}$ be the subgraphs of $\Gamma_t(R/(x)$ associated to the sets $D(\Sigma_1), \dots, D(\Sigma_s)$ respectively. We show that the associated graphs are the connected components of $\Gamma_t(R/(x))$.
Let  $\a_i = \bigcap_{\p \in \Sigma_i}\p$. From Corollary \ref{cor:ConnectedImpSigmaSigma'} and its proof we know that $\Sigma_i^{'} \cong \Gamma_t\left(R/(\a_i + (x))\right)$ is also connected for each $i$.

From Lemma \ref{lem:SimaSigma'0VrtxCmmn} we know that for distinct $i$ and $j$, the graphs $\Sigma_i^{'}$ and $\Sigma_j^{'}$ do not share vertices. Thus they are distinct connected subgraphs of $\Gamma_t(R/(x))$.

It remains to show that for every pair of distinct $\Sigma_i^{'}$ and $\Sigma_j^{'}$ there are no edges between them, so they are indeed the connected components of $\Gamma_t(R/(x))$.

For $i \neq j$, suppose there is an edge between $\q_1 /(x) \in \Sigma_{i}^{'}$ and $\q_2 /(x) \in \Sigma_{j}^{'}$. Let $S$ be the set of vertices of $\Gamma_t(R)$ in $\Sigma_i$ and let $T$ be the set of vertices of $\Gamma_t(R)$ which are not in $\Sigma_i$.
Note that $(S, T)$ is a partition of $\Min(R)$. Let $I$ and $J$ be the intersection of all the elements of $S$ and $T$ respectively. Take $\p_1$ and $\p_2$ such that $\q_1 \in D(\p_1)$ and $\q_2 \in D(\p_2)$. Since $I + J \subseteq \p_1 + \p_2 + (x)$, we have that $\height(I + J) \leq \height(\p_1 + \p_2 + (x))$. Suppose that the  equality holds. Take a minimal prime $\q$ of $\p_1 + \p_2 + (x)$ such that $\Ht(\q) = \Ht(\p_1 + \p_2 + (x))$. Since $I + J$ and $\p_1 + \p_2 +(x)$ have the same height, $\q$ must also be a minimal prime of $I + J$. This is not possible by  Lemma \ref{lem:Tor}, because $x \in \q$. We have that
\begin{eqnarray*}
\height(I + J) + 1 & \leq & \height( \mathfrak{p}_1 + \mathfrak{p}_2 + (x)) \\
& \leq & \height(\mathfrak{q}_1 + \mathfrak{q}_2) \\
& = & \height(\mathfrak{q}_1 /(x) + \mathfrak{q}_2 /(x)) + 1.
\end{eqnarray*}
Thus $\height(I + J) \leq \height(\mathfrak{q}_1 /(x) + \mathfrak{q}_2 /(x)) \leq t$. From the proof of Proposition \ref{pr:GraphConnectedPartitionEquiv} we know this means there is an edge between some prime in $S$ and some prime in $T$. Then, there is an edge between a vertex of $\Sigma_i$ and a vertex of another connected component of $\Gamma_t(R)$,  which is a contradiction.
We conclude that $\Sigma_{1}^{'}, \Sigma_{2}^{'}, \dots, \Sigma_{s}^{'}$ are the connected components of $\Gamma_t(R/(x))$.

\end{proof}

\section{Applications to Groebner deformations}

In this section we apply the results regarding a parameter to initial ideals and square free Groebner deformations.

\begin{remark}\label{RemGrad}
Let  $I\subseteq S=K[x_1,\ldots, x_n]/I$, where $I$ homogeneous under a grading given by a vector $w\in\ZZ_{>0}$.
Let $\m=(x_1,\ldots,x_n)$ be the maximal homogeneous ideal. Let $R=S/I$.
Let $J_{k}=S_{\geq t}$. Let $\widehat{R}$ denote the localization of $R$ with respect to $\{J_k\}$. We note that $\widehat{R}$ is equal to the $|m$-adic completion of $R$ and $R_\m$.
Let $\p$ be a homogeneous  prime ideal of $R$. We note that $\p \widehat{R}$ is a prime ideal in $\widehat{R}$, because 
$$\bigoplus_{k\in\NN} \frac{J_k (\widehat{R}/\p\widehat{R})}{J_{k+1} (\widehat{R}/ \p\widehat{R})}=R/\p$$ is a domain. As a consequence, there is a one-to-one  correspondence between the minimal primes of $R$ and the minimal primes of $\widehat{R}$, and their sums have the same height.
Then, $\Gamma_k(R)=\Gamma_k(\widehat{R}).$
\end{remark}

\begin{theorem}\label{ThmGroeb}
Let $S=K[x_1,\ldots,x_n]$ be a polynomial ring over a field.
Let $I$ be an  equidimensional ideal and $<$ a monomial order such that $\IN_<(I)$ is square-free.
Then, 
$$
\#\Gamma_t(S/I)=\#\Gamma_t(S/\IN(I)).
$$
As a consequence,
$\cdim(S/I)=\cdim(S/\IN(I))$.
\end{theorem}
\begin{proof}
Let $\eta=(x_1,\ldots, x_n)$.
There exists a vector $w\in\NN^n$ such that $\IN_<(I)=\IN_w(I)$ \cite[Proposition 1.11]{Stu}.
Let $A=K[t]$ be a polynomial ring, $L=\Frac(A)$, and $T=A\otimes_K S$.
We consider $T$ as a graded ring with $\deg(t)=1$, and $\deg(x_i)=w_i$ for every $i\in\NN$.
Given $f=\sum_\alpha c_\alpha x^{\alpha_1}_1\cdots x^{\alpha_n}_n \in S$, we take $f^w=t^{\deg_w f}\sum_\alpha c_\alpha\left(\frac{x_1}{t}\right)^{\alpha_1}\cdots  \left(\frac{x_n}{t}\right)^{\alpha_n}\in T$. 
We set $I^w=(f^w\; | \; f\in I)\subseteq T$ the homogenization of $I$, and $R=T/ I^w$.
Then, $R/tR \cong S/\IN_<(I)$ and  $R/(t-1)R \cong S/I$.
Let $\p_1,\ldots,\p_\ell$ be the minimal primes of $I$.
We note that $\p^w_1,\ldots,\p^w_\ell$ be the minimal primes of $I^w$ \cite[Lemma 2.3]{MV09}.
We note that $\dim T/\p^w_i=\dim S/\p_i+1=\dim S/I +1$ \cite[Lemma 2.3]{MV09}, and so, $I^w$ is equidimensional.
Since  $I^w$ is homogeneous, we have that $\p^w_i$ are also homogeneous. Thus, $\p^w_i+\p^w_j$ and its minimal primes are homogeneous for every $i,j$. We conclude that all these ideals are contained in $\n=(t,x_1,\ldots, x_n)$.
Hence,
\begin{align*}
\#\Gamma_k(S/\IN(I))&=\#\Gamma_k(R/tR) \hbox{ because }S/\IN(I)\cong R/tR\\
&=\#\Gamma_k(\widehat{R_\n}/t\widehat{R_\n})\hbox{ by Remark \ref{RemGrad}}\\
&=\#\Gamma_k(\widehat{R_\n})\hbox{ by Theorem \ref{ThmGraphParameter}}\\
&=\#\Gamma_k(R)\hbox{ by Remark \ref{RemGrad}}
\end{align*}

We now show that $\#\Gamma_k(R)=\# \Gamma_k(S/I)$. We have that 
\begin{align*}
\Ht(\p^w_i+\p^w_j)&\leq \Ht((\p_i+\p_j)^w) \hbox{ because }\p^w_i+\p^w_j\subseteq  (\p_i+\p_j)^w\\
& =\Ht(\p_i+\p_j) \hbox{  \cite[Proof of Lemma 2.3(6)]{MV09}}.
\end{align*}
Let $\overline{R}=R/(t-1)R$. Then, 
\begin{align*}
\Ht(\p_i+\p_j)&=\Ht(\p^w_i\overline{R}+\p^w_j\overline{R})\hbox{ because }S/I\cong \overline{R}\\
&=\Ht(\p^w_i+\p^w_j+(t-1))-1  \hbox{ because }T\hbox{ is a polynomial ring}\\
&\leq \Ht(\p^w_i+\p^w_j)\hbox{ by Krull's principal ideal Theorem.}
\end{align*}
We conclude that $\Ht(\p^w_i+\p^w_j))= \Ht(\p_i+\p_j)$. 
Hence, $\Gamma_k(S/I)\cong\Gamma_k(T/I^w)$.
We conclude $\#\Gamma_k(S/I)=\#\Gamma_k(S/\IN_<(I))$.
\end{proof}

It is known that the graphs $\Gamma_t(S/I)$ may vary after a field extension. Furthermore, the number of connected components might change. For instance, if $S=\RR[x,y]$ and $I=(x^2+y^2)$, $\Gamma_0(S/I)$ has only one vertex and it is connected. In contrast, $\Gamma_0(S/I\otimes_\RR \CC)$ has two vertices and it is disconnected.
As a consequence of Theorem \ref{ThmGroeb}, we obtain that $\# \Gamma_T(S/I)$ does not change when the field is extended if $\IN_<(I)$ is square-free.

\begin{corollary}\label{CorField}
Let $S=K[x_1,\ldots,x_n]$ be a polynomial ring over a field.
Let $I$ be an  equidimensional ideal and $<$ a monomial order such that $\IN_<(I)$ is square-free.
Then, 
$$
\#\Gamma_t(S/I)=\#\Gamma_t(S/I\otimes_K L)
$$
for every field extension $K\subseteq L$.
\end{corollary}
\begin{proof}
We observe that the initial ideal of $I\otimes_K L$ is  $\IN_<(I)\otimes_K L$, which is also square-free. 
Since the minimal primes of a square-free monomial ideal, and their sums, are ideals generated by variables, their heights are independent of the field.
Then,
\begin{align*}
\#\Gamma_t(S/I)&=\#\Gamma_t(S/\IN(I)) \hbox{ by Theorem \ref{ThmGroeb}}\\
&=\#\Gamma_t((S/\IN(I))\otimes_K L)\\
&=\#\Gamma_t((S\otimes_K L/\IN(I)\otimes_K L)\\
&=\#\Gamma_t((S\otimes_K L/I\otimes_K L)))  \hbox{ by Theorem \ref{ThmGroeb}}.
\end{align*}
\end{proof}

\section{Applications to Lyubeznik numbers}

Lyubeznik \cite{LyuDmod}  defined numerical inviariants for local rings in equal characteristic. We now recall their definition.

\begin{definition}[{\cite{LyuDmod}}] \label{LyuDef}
Let $(R, \m, K)$ be a local ring containing a field. Then, $\widehat{R}$  admits a surjection,  $\pi:S\to R$, from a regular local ring $(S,\eta,K)$ 
containing a field. 
Let  $n=\dim(S)$ and $I = \ker (\pi).$
The $i,j$-Lyubeznik number of $R$ is defined by
$$
\lambda_{i,j}(R):=\dim_K \Ext^i_S\left(K,H^{n-j}_I (S)\right),
$$
\end{definition}

We recall that the previous numbers  depends only on $R$, $i$, and $j$. In particular,  this number is independent of 
the choice of $S$ and of $\pi$ \cite[Lemma 4.3]{LyuDmod}.

Let $X$ be an equidimensional projective variety of dimension $d$ over a field $K$. 
By choosing an embedding $X \hookrightarrow \PP^n_K$, we can write $X=\Proj(R/I)$, where $I$ is a homogeneous ideal of the polynomial ring 
 $S=K[x_0,\ldots,x_n]$. 
 Let $\m = ( x_0, \ldots, x_n )$ be the homogeneous maximal ideal of $S$, and $R = ( S/I )_\m$ the local ring at the vertex of the affine cone over $X$.  
 
 Lyubeznik asked whether the Lyubeznik numbers $\lambda_{i,j}(R)$ are independent of $n$, and the choice of embedding of $X$ into $ \PP^n_K$, and so, we can write it as $\lambda_{i,j}(X)$. 
The question has been answered affirmatively for all Lyubeznik numbers by Zhang when $K$ has prime characteristic \cite[Theorem 1.1]{WZ-projective}. In contrast, these numbers may vary with the embedding in characteristric zero \cite[Theorem 1]{RSW18}.

The highest Lyubeznik number $\lambda_{d+1, d+1}(R)$ is independent of the choice of embedding  \cite[Theorem 2.7]{Z}, and it is 
well known that  $\lambda_{0,1}(R)$ is as well \cite[Proposition 3.1]{W}.  
It is also know that  $\lambda_{1,2}(R)$ is also an invariant in all characteristics.
These numbers  $\lambda_{0,1}(R), \lambda_{1,2}(R)$, and $\lambda_{d+1, d+1}(R)$
are defined in terms of a geometric version of  $\Gamma_i(R)$.

\begin{definition}
For $X$ an equidimensional projective variety of dimension $d$, given an integer
 $1 \leq t \leq d$, define the graph $\Gamma_t(X)$ as follows:
\begin{enumerate}[label=\textup{(\arabic*)}]
\item The vertices of $\Gamma_t(X)$ are indexed by the irreducible components of $X$,  and 
\item There is an edge between distinct vertices $Z$ and $W$ if and only if 
\[\dim( Z \cap W) \geq d - t.\]
\end{enumerate}
\end{definition}

\begin{theorem}\label{ThmLyuGraphs}
Let $X$ be an equidimensional projective variety of dimension $d$ over a field $K$. 
Then,
\begin{itemize}
\item  $\lambda_{0,1}(X)= \#\Gamma_d (X\otimes_K \overline{K})-1$ \cite[Proposition 3.1]{W};
\item $ \lambda_{1,2}(X)= \#\Gamma_{d-1} (X\otimes_K \overline{K})-\#\Gamma_d (X\otimes_K \overline{K})$ \cite[Theorem 7.4]{NBSW};
\item  $\lambda_{d+1, d+1}(X)=\# \Gamma_2(X\otimes_K \overline{K})$ \cite[Theorem 2.7]{Z};
\end{itemize}
\end{theorem}

As a consequence of Theorem \ref{thm:Main2}, we  provide a way to compute certain Lyubeznik numbers from quare-free initial ideals. This type of questions was previously studied by  Nadi and Varbaro \cite{NV}. In particular, our next theorem extends one of their results  \cite[Proposition 2.11]{NV}.  

\begin{theorem}\label{ThmLyuNum}
Let $S=K[x_1,\ldots,x_n]$ be a polynomial ring over a field.
Let $I$ be an  equidimensional ideal and $<$ a monomial order such that $\IN_<(I)$ is square-free.
Let $X=\Proj(S/I)$ and $Y=\Proj(S/\IN_<(I))$.
Then, 
\begin{itemize}
\item  $\lambda_{0,1}(X)= \lambda_{0,1}(Y)$; 
\item $ \lambda_{1,2}(X)=\lambda_{1,2}(Y) $;
\item  $\lambda_{d+1, d+1}(X)=\lambda_{d+1, d+1}(Y)$.
\end{itemize}
\end{theorem}
\begin{proof}
We have that 
$\Gamma_t(X)=\Gamma((R/I)_\m)$ and $\Gamma_t(Y)=\Gamma((S/\IN_<(I))_\m)$ \cite[Lemma 7.3]{NBSW}.
Then, the result follows from Theorems \ref{ThmGroeb} and \ref{ThmLyuGraphs}.
\end{proof}

\section*{Acknowledgments}

We thank Raul G\'omez M\'u\~{n}oz for  valuable comments.

\bibliographystyle{alpha}
\bibliography{References}
\end{document}